\documentclass[11pt,a4paper]{amsart}

\usepackage{fancyhdr,euler,stmaryrd} 
\usepackage{amssymb} 
\usepackage{combelow} 
\usepackage{tensor} 
\usepackage{hyperref} 
\usepackage[all]{xy} 
\usepackage{mathrsfs} 
\usepackage{indentfirst} 
\usepackage{mathtools} 
\usepackage{titlesec} 

\newtheorem{theorem}{Theorem}[section]

\newtheorem{lemma}[theorem]{Lemma}
\newtheorem{proposition}[theorem]{Proposition}
\theoremstyle{definition}
\newtheorem{definition}[theorem]{Definition}
\newtheorem{subsec}[theorem]{}
\renewenvironment{proof}{\par\noindent\textbf{Proof:}}{\hfill $\blacksquare$\par\smallskip}

\newcommand{\G}{\bar{G}}
\newcommand{\C}{\mathcal{C}}
\renewcommand{\O}{\mathcal{O}}
\newcommand{\Aall}{\mathbf{A}}
\newcommand{\Apall}{\mathbf{A'}}
\newcommand{\Ball}{\mathbf{B}}
\newcommand{\Call}{\mathbfcal{C}} \DeclareMathAlphabet\mathbfcal{OMS}{cmsy}{b}{n}
\newcommand{\Gall}{\mathbf{\G}}
\newcommand{\Mall}{\mathbf{\tilde{M}}}
\newcommand{\op}{^\mathrm{op}}

\newcommand{\Hom}[3]{\mathrm{Hom}_{#1}(#2,#3)}
\newcommand{\set}[1]{\left\{#1\right\}}

\newcommand{\titlename}	
{Group graded Morita equivalences\\for wreath products}
\newcommand{\shorttitlename}
{Group graded Morita equivalences for wreath products}

\newcommand{\authorname}      {Virgilius-Aurelian Minu\cb{t}\u{a}}
\newcommand{\shortauthorname} {V. A. Minu\cb{t}\u{a}}
\newcommand{\universityname}  {Babe\cb{s}-Bolyai University of Cluj-Napoca}
\newcommand{\facultyname}     {Faculty of Mathematics and Computer Science}
\newcommand{\departmentname}  {Department of Mathematics}
\renewcommand{\email}         {minuta.aurelian@math.ubbcluj.ro}

\newcommand{\articleabstract}{Starting with group graded Morita equivalences, we obtain Morita equivalences for tensor products and wreath products.}

\newcommand{\msc}{16W50, 20E22, 20C05, 20C20, 16D90, 16S35}

\renewcommand{\keywords}{Group graded algebras, wreath products, Morita equivalences, crossed products, centralizer subalgebra}

\hypersetup{colorlinks = true, linkcolor = blue, anchorcolor =red, citecolor = blue, filecolor = red, urlcolor = blue, pdfauthor=author}
\titleformat{\section}{\Large\bfseries}{\thesection}{1em}{}
\fancypagestyle{mypagestyle}{
  \fancyhf{}
  \fancyhead[OC]{\textit{\shorttitlename}}
  \fancyhead[EC]{\textit{\shortauthorname}}
  \fancyfoot[C]{\medskip $\leftslice\,$\thepage$\,\rightslice$ }
  
}
\fancypagestyle{firstpagestyle}{
  \fancyhf{}
  \fancyfoot[C]{\medskip $\leftslice\,$\thepage$\,\rightslice$ }
  
}
\title[\shorttitlename]{\LARGE{\titlename}}
\newcommand{\institution}{
\universityname\\
\facultyname\\
\departmentname}
\author[\shortauthorname]{\Large{\authorname}
\medskip\\
{\footnotesize \institution\\
email: \texttt{\href{mailto:\email}{\email}}
}}
\begin{document}
\begin{abstract}
\articleabstract\\[0.1cm]
\textsc{MSC 2010.} \msc\\[0.1cm]
\end{abstract}
\begingroup
\def\uppercasenonmath#1{} 
\let\MakeUppercase\relax 
\maketitle
\endgroup
\thispagestyle{firstpagestyle}

\section{Introduction} In this article, we  continue the study done in \cite{art:MM1}, \cite{art:MM2} and \cite{art:MVA1}, and we obtain group graded Morita equivalences for tensor products (Proposition \ref{prop:tensorProduct_Morita}) and wreath products (Theorem \ref{th:main}). The main motivation for such constructions in the representation theory of finite groups is given by the fact that in order to prove most reduction theorems, recent results of Britta Sp\"ath, surveyed in \cite{art:Spath2017}, \cite{ch:Spath2017} and \cite{ch:Spath2018}, show that a new character triple can be constructed via a wreath product construction of character triples (\cite[Theorem 2.21]{ch:Spath2018}). There is a link between character triples and group graded Morita equivalences, presented in \cite{art:MM2}, so we want to prove that a similar wreath product construction can also be made for the corresponding group graded Morita equivalences.

More precisely, in Theorem 6.7 of \cite{art:MM2}, it is proved that certain character triples relations utilized by Britta Sp\"ath: the first-order relation (\cite[Definition 2.1]{ch:Spath2018}) and the central-order relation (\cite[Definition 2.7]{ch:Spath2018}), are consequences of a special type of group graded Morita equivalences induced by a graded bimodule over a $G$-graded $G$-acted algebra (usually denoted by $\C$ as in Section \ref{subsec:defs}), where $G$ is a finite group. More details about group graded Morita theory over $\C$ can be found in \cite{art:MVA1}.

Another motivation comes from the fact that it is already known  by \cite[Theorem 5.1.21]{book:Marcus1999} that  Morita equivalences can be extended to wreath products.

This paper is organized as follows: In Section 2, we introduce the general notations and we recall from \cite{art:MM2} the definitions of a $G$-graded $G$-acted algebra, of a $G$-graded algebra over $\C$, of a $G$-graded bimodule over $\C$ and the notion of a $G$-graded Morita equivalence over $\C$.
In Section 3, we prove that the previously recalled algebraic constructions are compatible with tensor products and the main proposition in this section, Proposition \ref{prop:tensorProduct_Morita}, proves that the tensor products of some group graded Morita equivalent algebras over some group graded group acted algebras remain group graded Morita equivalent over a group graded group acted algebra. In Section 4, we prove that the previously enumerated algebra types are also compatible with wreath products. Finally, in  Section 5, our main result, Theorem \ref{th:main}, proves that the wreath product between a $G$-graded bimodule over $\C$ and $S_n$ (the symmetric group of order $n$) is also a group graded bimodule over $\C^{\otimes n}$, and moreover, if this bimodule induces a $G$-graded Morita equivalence over $\C$, then its wreath product with $S_n$ will induce a group graded Morita equivalence over $\C^{\otimes n}$.
\medskip
\medskip
\medskip
\medskip
\section{Preliminaries} \label{sec:Preliminaries}

\begin{subsec}\label{subsec:Hypo}
All rings in this paper are associative with identity $1 \neq 0$ and all modules are left (unless otherwise specified) unital and finitely generated. Throughout this article $n$ will represent an arbitrary nonzero natural number, and $\O$ is a commutative ring.
\end{subsec}

\begin{subsec} Let $G$ be a finite group and $N$ a normal subgroup of $G$. We denote by $\G:=G/N$.

Note that most results in this paper will utilize ``$\G$-gradings'', although this is not essential: one may consider instead the gradings to be given directly by $G$. The reasoning behind this particular choice is to match our notations previously used in articles \cite{art:MM1} and \cite{art:MM2}, given that our main application for the results of this project is the strongly $\G$-graded algebra $A=b\O G$, where $b$ is a $\G$-invariant block of $\O N$.
\end{subsec}

\begin{subsec}\label{subsec:defs} We recall from \cite{art:MM2} the following definitions:
\begin{definition}
  An algebra $\C$ is a \textit{$\G$-graded $\G$-acted algebra} if
 \begin{enumerate}
	\item[(1)] $\C$ is $\G$-graded, and we write $\C=\bigoplus_{\bar{g}\in \bar G}\C_{\bar{g}}$;
	\item[(2)] $\G$ acts on $\C$ (always on the left in this article);
	\item[(3)] for all $\bar{g},\bar{h}\in \G$ and for all $c\in \C_{\bar{h}}$ we have  $\tensor*[_{}^{\bar{g}}]{c}{_{}^{}}\in \C_{\tensor*[_{}^{\bar{g}}]{\bar{h}}{_{}^{}}}$.
 \end{enumerate}
\end{definition}
\begin{definition} \label{def:algebraOverC}
Let $\C$ be a $\G$-graded $\G$-acted algebra. We say the $A$ is a $\G$-graded algebra over $\C$   if there is a  $\G$-graded $\G$-acted algebra homomorphism
\[\zeta:\C\to C_A(B),\] where $B:=A_1$ and $C_A(B)$ is the centralizer of $B$ in $A$,  i.e. for any $\bar{h}\in \G$ and $c\in \C_{\bar{h}}$, we have $\zeta(c)\in C_A(B)_{\bar{h}}$, and for every $\bar{g}\in\G$, $\zeta(\tensor*[^{\bar{g}}]{c}{})=\tensor*[^{\bar{g}}]{\zeta(c)}{}$.
\end{definition}
\begin{definition}
 Let $A$ and $A'$ be two $\G$-graded crossed products over a $\G$-graded $\G$-acted algebra $\C$, with structure maps $\zeta$ and $\zeta'$, respectively.

{\rm a)}	We say that $\tilde{M}$ is a $\G$-graded $(A,A')$-bimodule over $\C$ if:
	\begin{enumerate}
		\item[(1)] $\tilde{M}$ is an $(A,A')$-bimodule;
		\item[(2)] $\tilde{M}$ has a decomposition $\tilde{M}=\bigoplus_{\bar{g}\in\G}\tilde{M}_{\bar{g}}$ such that $A_{\bar{g}}\tilde{M}_{\bar{x}}A'_{\bar{h}}\subseteq \tilde{M}_{\bar{g}\bar{x}\bar{h}}$, for all $\bar{g}, \bar{x},\bar{h}\in \G$;
		\item[(3)] $\tilde{m}_{\bar{g}} c=\tensor*[^{\bar{g}}]{c}{} \tilde{m}_{\bar{g}}$, for all $c\in \C$, $\tilde{m}_{\bar{g}}\in\tilde{M}_{\bar{g}}$, $\bar{g}\in \G$, where $c\tilde {m} = \zeta(c)\tilde{m}$ and $\tilde{m}c=\tilde{m}\zeta'(c)$, for all $c\in \C$, $\tilde{m}\in\tilde{M}$.
	\end{enumerate}

{\rm b)} $\G$-graded $(A,A')$-bimodules over $\C$ form a category,
where the morphisms between $\G$-graded $(A,A')$-bimodules over $\C$ are just homomorphisms between $\G$-graded $(A,A')$-bimodules.
\end{definition}
\begin{definition}
Let $A$ and $A'$ be two $\G$-graded crossed products over a $\G$-graded $\G$-acted algebra $\C$, and let $\tilde{M}$ be a $\G$-graded $(A,A')$-bimodule over $\C$. Clearly, the $A$-dual $\tilde{M}^{\ast}=\Hom{A}{\tilde{M}}{A}$ of $\tilde{M}$ is a $\G$-graded $(A',A)$-bimodule over $\C$. We say that  $\tilde{M}$ induces a $\G$-graded Morita equivalence over $\C$ between $A$ and $A'$, if $\tilde{M}\otimes_{A'}\tilde{M}^{\ast}\simeq A$ as $\G$-graded $(A,A)$-bimodules over $\C$ and if $\tilde{M}^{\ast}\otimes_{A}\tilde{M}\simeq A'$ as $\G$-graded $(A',A')$-bimodules over $\C$.
\end{definition}
\end{subsec}
\medskip
\medskip
\medskip
\medskip
\section{Tensor products}
\begin{subsec}
Consider $G_i$ to be a finite group, $N_i$ to be a normal subgroup of $G_i$ and denote by $\G_i=G_i/N_i$, for all $i\in\set{1,\ldots,n}$. We denote by $$\Gall:=\prod_{i=1}^{n}{\G_i}.$$

\begin{lemma} \label{lemma:tensorProduct_over_C}
Let $A_i$ be $\G_i$-graded algebras and $\C_i$ be $\G_i$-graded $\G_i$-acted algebras, for all $i\in\set{1,\ldots,n}$. The following affirmations hold:
	\begin{enumerate}
		\item [(1)] The tensor product $\Aall:=A_1\otimes\ldots\otimes A_n$ is a $\Gall$-graded algebra;
		\item [(2)] If $A_i$ are $\G_i$-graded crossed products, for all $i\in\set{1,\ldots,n}$, then $\Aall$ is a $\Gall$-graded crossed product;
		\item [(3)] The tensor product $\Call:=\C_1\otimes\ldots\otimes\C_n$ is a $\Gall$-graded $\Gall$-acted algebra;
		\item [(4)] If $A_i$ are $\G_i$-graded algebras over $\C_i$, for all $i\in\set{1,\ldots,n}$, then $\Aall$ is a $\Gall$-graded algebra over $\Call$.
	\end{enumerate}
\end{lemma}

\begin{proof}
(1) It is clear that $\Aall$ is a $\Gall$-graded algebra, with the $(g_1,\ldots,g_n)$-component
$$\Aall_{(g_1,\ldots,g_n)}:=A_{1,g_1}\otimes\ldots\otimes A_{n,g_n},$$
where $A_{i,g_i}$ is the $g_i$-component of $A_i$, for all $g_i\in \G_i$ and for all $i$. 

(2) Choose invertible homogeneous elements $u_{i,g}$ in each $A_{i,g}$, for all $g\in \G_i$ and for all $i$. Thus, for each $(g_1,\ldots,g_n)\in \Gall$ the homogeneous element
$$u_{(g_1,\ldots,g_n)}:=u_{1,g_1}\otimes\ldots\otimes u_{n,g_n}\in \Aall_{(g_1,\ldots,g_n)}$$
is clearly invertible.

(3) The $\Gall$-grading of $\Call$ is a given by (1).  The action of $\Gall$ on $\Call$ is defined by
$$\tensor*[^{(g_1,\ldots,g_n)}]{(a_1\otimes\ldots\otimes a_n)}{}:=\tensor*[^{g_1}]{a}{_1}\otimes\ldots\otimes \tensor*[^{g_n}]{a}{_n},$$
for all $(g_1,\ldots,g_n) \in \Gall$ and $a_1\otimes\ldots\otimes a_n\in\Call$.
It is easy too see that  for all $(g_1,\ldots,g_n),(h_1,\ldots,h_n)\in \Gall$ and for all $a_1\otimes\ldots\otimes a_n\in \Call_{(h_1,\ldots,h_n)} $ we have
$$\tensor*[^{(g_1,\ldots,g_n)}]{(a_1\otimes\ldots\otimes a_n)}{}\in \Call_{\tensor*[^{(g_1,\ldots,g_n)}]{(h_1,\ldots,h_n)}{}}.$$

(4) By part (1), the identity component of $\Aall$ is $\Ball=B_1\otimes\ldots\otimes B_n$, where $B_i$ is the identity component of $A_i$, for all $i$.

By the assumptions, we have the $\G_i$-graded $\G_i$-acted structure homomorphisms
$$\zeta_i:\C_i\to C_{A_i}(B_i),$$
for all $i$. We define $\zeta:\Call\to C_{\Aall}(\Ball)$ by
$$\zeta(a_1\otimes\ldots\otimes a_n)=\zeta_1(a_1)\otimes\ldots\otimes\zeta_n(a_n),$$
for all $a_1\otimes\ldots\otimes a_n\in \Call$. It is easy to prove that $\zeta$ verifies the conditions of Definition \ref{def:algebraOverC}.
\end{proof}

\begin{proposition} \label{prop:tensorProduct_Morita}
Assume that $\C_i$ are $\G_i$-graded $\G_i$-acted algebras and that $A_i$ and $A'_i$ are $\G_i$-graded crossed products over $\C_i$, for all $i\in\set{1,\ldots,n}$. If $A_i$ and $A'_i$ are $\G_i$-graded Morita equivalent over $\C_i$, and if $\tilde{M}_i$ is a $\G_i$-graded $(A_i,A'_i)$-bimodule over $\C_i$, that induces the said equivalence, for all $i$, then:
\begin{enumerate}
	\item [(1)] $\Mall:=\tilde{M}_1\otimes\ldots\otimes\tilde{M}_n$ is a $\Gall$-graded $(\Aall,\Apall)$-bimodule over $\Call$, where $\Aall:=A_1\otimes\ldots\otimes A_n$, $\Apall:=A'_1\otimes\ldots\otimes A'_n$ and $\Call:=\C_1\otimes\ldots\otimes\C_n$;
	\item [(2)] $\Mall$ induces a $\Gall$-graded Morita equivalence over $\Call$ between $\Aall$ and $\Apall$.
\end{enumerate}
\end{proposition}
\begin{proof}
(1) By Lemma \ref{lemma:tensorProduct_over_C},  $\Aall$ and $\Apall$ are $\Gall$-graded crossed products over $\Call$.
Obviously, $\Mall$ is a $\Gall$-graded  $(\Aall,\Apall)$-bimodule with the $(g_1, \ldots, g_n)$-component
$$\Mall_{(g_1,\ldots,g_n)}:=\tilde{M}_{1,g_1}\otimes\ldots\otimes \tilde{M}_{n,g_n},$$
where $\tilde{M}_{i,g_i}$ is the $g_i$-component of $\tilde{M}_i$, for all $g_i\in \G_i$ and for all $i$. It is also clear that
$$(\tilde{m}_{1,g_1}\otimes\ldots\otimes\tilde{m}_{n,g_n})(c_1\otimes\ldots\otimes c_n) = \tensor*[^{g}]{{(c_1\otimes\ldots\otimes c_n)}}{}(\tilde{m}_{1,g_1}\otimes\ldots\otimes\tilde{m}_{n,g_n}),$$
for all $\tilde{m}_{1,g_1}\otimes\ldots\otimes\tilde{m}_{n,g_n}\in\Mall_{(g_1,\ldots,g_n)}$ and $c_1\otimes\ldots\otimes c_n\in\Call$ and for all $g=(g_1,\ldots,g_n)\in\Gall$. 

(2) It remains to prove that
$$
\Mall\otimes_{\Apall}(\Mall)^{\ast}\simeq \Aall\text{ as }\Gall\text{-graded }(\Aall,\Aall)\text{-bimodules over }\Call,
$$
and that
$$
(\Mall)^{\ast}\otimes_{\Aall}\Mall\simeq \Apall\text{ as }\Gall\text{-graded }(\Apall,\Apall)\text{-bimodules over }\Call.
$$

We will only check the first isomorphism:
$$
\begin{array}{rcl}
\Mall\otimes_{\Apall}(\Mall)^{\ast} & = & (\tilde{M}_1\otimes\ldots\otimes\tilde{M}_n)\otimes_{\Apall}(\tilde{M}_1\otimes\ldots\otimes\tilde{M}_n)^{\ast}\\
 & \simeq & (\tilde{M}_1\otimes\ldots\otimes\tilde{M}_n)\otimes_{\Apall}(\tilde{M}_1^{\ast}\otimes\ldots\otimes\tilde{M}_n^{\ast})\\
 & = & (\tilde{M}_1\otimes\ldots\otimes\tilde{M}_n)\otimes_{A'_1\otimes\ldots\otimes A'_n}(\tilde{M}_1^{\ast}\otimes\ldots\otimes\tilde{M}_n^{\ast})\\
 & \simeq & (\tilde{M}_1\otimes_{A'_1}\tilde{M}_1^{\ast})\otimes\ldots\otimes(\tilde{M}_n\otimes_{A'_n}\tilde{M}_n^{\ast})\\
 & \simeq & A_1\otimes\ldots\otimes A_n \quad = \quad \Aall,
\end{array}
$$
as $\Gall$-graded $(\Aall,\Aall)$-bimodules over $\Call$.
\end{proof}
\end{subsec}
\newpage
\section{Wreath products for algebras} \label{sec:Wreath_products}

Consider the notations from Section \ref{sec:Preliminaries}. We denote $\G^n:=\G\times\ldots\times\G$ ($n$ times). We recall the definition of a wreath product as in \cite[Definition 2.19]{ch:Spath2018} and \cite[Section 5.1.C]{book:Marcus1999}:
\begin{definition}
The wreath product $\G\wr S_n$ is the semidirect product $\G^n\rtimes S_n$, where $S_n$ acts on $\G^n$ (on the left) by permuting the components:
$$\tensor*[^{\sigma}]{{(g_1,\ldots,g_n)}}{}:=(g_{\sigma^{-1}(1)},\ldots,g_{\sigma^{-1}(n)})$$
More exactly, the elements of $\G\wr S_n$ are of the form $((g_1,\ldots,g_n),\sigma)$, and the multiplication is:
$$((g_1,\ldots,g_n),\sigma)((h_1,\ldots,h_n),\tau):=((g_1,\ldots,g_n)\cdot \tensor*[^{\sigma}]{{(h_1,\ldots,h_n)}}{},\sigma\tau),$$
for all $g_1,\ldots,g_n,h_1,\ldots,h_n\in \G$ and $\sigma,\tau\in S_n$.
\end{definition}

\begin{definition}
Let $A$ be an algebra. We denote by $A^{\otimes n}:=A\otimes\ldots\otimes A$ ($n$ times). The wreath product $A\wr S_n$ is 
 $$A\wr S_n := A^{\otimes n}\otimes \O S_n$$
 as $\O$-modules, with multiplication 
$$
\begin{array}{l}
((a_1\otimes\ldots\otimes a_n)\otimes\sigma)((b_1\otimes\ldots\otimes b_n)\otimes\tau)\\
\qquad := ((a_1\otimes\ldots\otimes a_n)\cdot\tensor[^{\sigma}]{(b_1\otimes\ldots\otimes b_n)}{})\otimes (\sigma\tau),
\end{array}
$$
where
$$\tensor[^{\sigma}]{(b_1\otimes\ldots\otimes b_n)}{}:=b_{\sigma^{-1}(1)}\otimes\ldots\otimes b_{\sigma^{-1}(n)},$$
for all $(a_1\otimes\ldots\otimes a_n)\otimes\sigma$, $(b_1\otimes\ldots\otimes b_n)\otimes\tau\in A\wr S_n$.
\end{definition}

\begin{lemma} \label{lemma:WR_algebra}
Let $A$ be a $\G$-graded algebra and $\C$ be a $\G$-graded $\G$-acted algebra. The following affirmations hold:
\begin{enumerate}
	\item [(1)] $A\wr S_n$ is a $\bar{G}\wr S_n$-graded algebra;
	\item [(2)] If $A$ is a $\G$-graded crossed product, then $A\wr S_n$ is a $\bar{G}\wr S_n$-graded crossed product;
	\item [(3)] $\C^{\otimes n}$ is a $\G\wr S_n$-acted $\G^n$-graded algebra;
	\item [(4)] If $A$ is a $\G$-graded algebra over $\C$, then $A\wr S_n$ is a $\G\wr S_n$-graded algebra over $\C^{\otimes n}$.
\end{enumerate}
\end{lemma}

\begin{proof}
(1)  The $((g_1,\ldots,g_n),\sigma)$-component of $A\wr S_n$ is
$$(A\wr S_n)_{((g_1,\ldots,g_n),\sigma)}:=(A_{g_1}\otimes\ldots\otimes A_{g_n})\otimes\O\sigma,$$
for each $((g_1,\ldots,g_n),\sigma)\in \bar{G}\wr S_n$. Indeed,
\begingroup
\allowdisplaybreaks
\begin{align*}
   &(A\wr S_n)_{((g_1,\ldots,g_n),\sigma)}(A\wr S_n)_{((h_1,\ldots,h_n),\tau)}\\
 &\qquad =  ((A_{g_1}\otimes\ldots\otimes A_{g_n})\otimes\O\sigma)((A_{h_1}\otimes\ldots\otimes A_{h_n})\otimes\O\tau)\\
 &\qquad =  ((A_{g_1}\otimes\ldots\otimes A_{g_n})\cdot \tensor[^{\sigma}]{{(A_{h_1}\otimes\ldots\otimes A_{h_n})}}{})\otimes(\O\sigma\otimes\O\tau)\\
 &\qquad =  ((A_{g_1}\otimes\ldots\otimes A_{g_n})\cdot (A_{h_{\sigma^{-1}{(1)}}}\otimes\ldots\otimes A_{h_{\sigma^{-1}{(n)}}}))\otimes\O(\sigma\tau)\\
 &\qquad =  (A_{g_1}A_{h_{\sigma^{-1}{(1)}}}\otimes\ldots\otimes A_{g_n}A_{h_{\sigma^{-1}{(n)}}})\otimes\O(\sigma\tau)\\
 &\qquad \subseteq  (A_{g_1 h_{\sigma^{-1}{(1)}}}\otimes\ldots\otimes A_{g_n h_{\sigma^{-1}{(n)}}})\otimes\O(\sigma\tau)\\
 &\qquad =  (A\wr S_n)_{(((g_1,\ldots,g_n)\cdot \tensor*[^{\sigma}]{{(h_1,\ldots,h_n)}}{}),\sigma\tau)}\\
 &\qquad =  (A\wr S_n)_{((g_1,\ldots,g_n),\sigma)((h_1,\ldots,h_n),\tau)}.
\end{align*}
\endgroup

(2) We  choose invertible homogeneous elements $u_g\in A_g$ for all $g\in\G$. For  $((g_1,\ldots,g_n),\sigma)\in \G\wr S_n$ the  homogeneous element
$$u_{((g_1,\ldots,g_n),\sigma)}:=(u_{g_1}\otimes\ldots\otimes u_{g_n})\otimes\sigma,$$
is clearly invertible, with
$$u^{-1}_{((g_1,\ldots,g_n),\sigma)}:=(u^{-1}_{g_{\sigma(1)}}\otimes\ldots\otimes u^{-1}_{g_{\sigma(n)}})\otimes\sigma^{-1}.$$

(3) By Lemma \ref{lemma:tensorProduct_over_C}, we know that $\C^{\otimes n}$ is a $\G^n$-graded algebra. It remains to prove that it is $\G\wr S_n$-acted and that the action is compatible with the gradings. We define the action of $\G\wr S_n$ on $\C^{\otimes n}$ as follows:
$$
\begin{array}{l}
\tensor*[^{{((g_1,\ldots,g_n),\sigma)}}]{{(c_1\otimes\ldots\otimes c_n)}}{}:=\tensor*[^{{g_{1}}}]{{c_{\sigma^{-1}(1)}}}{}\otimes\ldots\otimes\tensor*[^{{g_{n}}}]{{c_{\sigma^{-1}(n)}}}{},
\end{array}
$$
where $((g_1,\ldots,g_n),\sigma)\in \G\wr S_n$ and $c_1\otimes\ldots\otimes c_n\in \C^{\otimes n}$. We have:
$$
\begin{array}{rcl}
\tensor*[^{{((1_{\G},\ldots,1_{\G}),e)}}]{{(a_1\otimes\ldots\otimes a_n)}}{} & = & \tensor*[^{{1_{\G}}}]{{a_1}}{}\otimes\ldots\otimes\tensor*[^{{1_{\G}}}]{{a_n}}{}\\
& = & a_1\otimes\ldots\otimes a_n,\\
\end{array}
$$
$$
\begin{array}{l}
\tensor*[^{{(((g_1,\ldots,g_n),\sigma)((h_1,\ldots,h_n),\tau))}}]{{(a_1\otimes\ldots\otimes a_n)}}{}\\
\qquad = \tensor*[^{{((g_1,\ldots,g_n)\cdot \tensor*[^{\sigma}]{{(h_1,\ldots,h_n)}}{},\sigma\tau)}}]{{(a_1\otimes\ldots\otimes a_n)}}{}\\
\qquad = \tensor*[^{{((g_1,\ldots,g_n)\cdot (h_{\sigma^{-1}(1)},\ldots,h_{\sigma^{-1}(n)}),\sigma\tau)}}]{{(a_1\otimes\ldots\otimes a_n)}}{}\\
\qquad = \tensor*[^{{((g_1 h_{\sigma^{-1}(1)},\ldots,g_n h_{\sigma^{-1}(n)}),\sigma\tau)}}]{{(a_1\otimes\ldots\otimes a_n)}}{}\\
\qquad = \tensor*[^{{g_1 h_{\sigma^{-1}(1)}}}]{{a_{(\sigma\tau)^{-1}(1)}}}{}\otimes\ldots\otimes\tensor*[^{{g_n h_{\sigma^{-1}(n)}}}]{{a_{(\sigma\tau)^{-1}(n)}}}{}\\
\qquad = \tensor*[^{{g_1 h_{\sigma^{-1}(1)}}}]{{a_{\tau^{-1}(\sigma^{-1}(1))}}}{}\otimes\ldots\otimes\tensor*[^{{g_n h_{\sigma^{-1}(n)}}}]{{a_{\tau^{-1}(\sigma^{-1}(n))}}}{}\\
\qquad = \tensor*[^{{((g_1,\ldots,g_n),\sigma)}}]{{(\tensor*[^{{h_{1}}}]{{a_{\tau^{-1}(1)}}}{}\otimes\ldots\otimes\tensor*[^{{h_{n}}}]{{a_{\tau^{-1}(n)}}}{})}}{}\\
\qquad = \tensor*[^{{((g_1,\ldots,g_n),\sigma)}}]{{\left(\tensor*[^{{((h_1,\ldots,h_n),\tau)}}]{{(a_1\otimes\ldots\otimes a_n)}}{}\right)}}{},
\end{array}
$$
$$
\begin{array}{l}
\tensor*[^{{((g_1,\ldots,g_n),\sigma)}}]{{((a_1\otimes\ldots\otimes a_n)\cdot(b_1\otimes\ldots\otimes b_n))}}{}\\
\qquad = \tensor*[^{{((g_1,\ldots,g_n),\sigma)}}]{{(a_1 b_1\otimes\ldots\otimes a_n b_n)}}{}\\
\qquad = \tensor*[^{g_{1}}]{(a_{\sigma^{-1}(1)} b_{\sigma^{-1}(1)})}{}\otimes\ldots\otimes \tensor*[^{g_{n}}]{(a_{\sigma^{-1}(n)} b_{\sigma^{-1}(n)})}{}\\
\qquad = \tensor*[^{g_{1}}]{{a_{\sigma^{-1}(1)}}}{}\cdot\tensor*[^{g_{1}}]{{b_{\sigma^{-1}(1)}}}{}\otimes\ldots\otimes \tensor*[^{g_{n}}]{{a_{\sigma^{-1}(n)}}}{}\cdot\tensor*[^{g_{n}}]{{b_{\sigma^{-1}(n)}}}{}\\
\qquad = (\tensor*[^{g_{1}}]{{a_{\sigma^{-1}(1)}}}{}\otimes\ldots\otimes \tensor*[^{g_{n}}]{{a_{\sigma^{-1}(n)}}}{})(\tensor*[^{g_{1}}]{{b_{\sigma^{-1}(1)}}}{}\otimes\ldots\otimes\tensor*[^{g_{n}}]{{b_{\sigma^{-1}(n)}}}{})\\
\qquad = \tensor*[^{{((g_1,\ldots,g_n),\sigma)}}]{{(a_1\otimes\ldots\otimes a_n)}}{}\cdot \tensor*[^{{((g_1,\ldots,g_n),\sigma)}}]{{(b_1\otimes\ldots\otimes b_n)}}{},
\end{array}
$$
and
$$
\begin{array}{rcl}
\tensor*[^{{((h_1,\ldots,h_n),\tau)}}]{{(c_1\otimes\ldots\otimes c_n)}}{} & = & \tensor*[^{{h_{1}}}]{{c_{\tau^{-1}(1)}}}{}\otimes\ldots\otimes\tensor*[^{{h_{n}}}]{{c_{\tau^{-1}(n)}}}{}\\
& \in & \C_{\tensor*[^{{h_{1}}}]{{g_{\tau^{-1}(1)}}}{}}\otimes\ldots\otimes\C_{\tensor*[^{{h_{n}}}]{{g_{\tau^{-1}(n)}}}{}}\\
& = & (\C^{\otimes n})_{(\tensor*[^{{h_{1}}}]{{g_{\tau^{-1}(1)}}}{},\ldots,\tensor*[^{{h_{n}}}]{{g_{\tau^{-1}(n)}}}{})}\\
& = & (\C^{\otimes n})_{\tensor*[^{{((h_1,\ldots,h_n),\tau)}}]{{(g_1,\ldots,g_n)}}{}},
\end{array}
$$
for all $a_1\otimes\ldots\otimes a_n$, $b_1\otimes\ldots\otimes b_n \in \C^{\otimes n}$, for all $c_1\otimes\ldots\otimes c_n \in \C^{\otimes n}_{(g_1,\ldots,g_n)}$ and for all $((g_1,\ldots,g_n),\sigma)$, $((h_1,\ldots,h_n),\tau)\in \G\wr S_n$.

(4) By assumption, there exists a $\G$-graded $\G$-acted algebra homomorphism $\zeta:\C\to C_A(B),$
where $B$ is the identity component of $A$. Henceforth, we have a $\G^{n}$-graded $\G^{n}$-acted algebra homomorphism:
$$\zeta^{\otimes n}:\C^{\otimes n}\to C_{A}(B)^{\otimes n}.$$
Now, via the identification: $$A^{\otimes n}\ni a_1\otimes\ldots\otimes a_n=(a_1\otimes\ldots\otimes a_n)\otimes e\in A\wr S_n,$$
we clearly have the following inclusion:
$$C_A(B)^{\otimes n}\subseteq C_{A\wr S_n}(B^{\otimes n})=C_{A\wr S_n}((A\wr S_n)_{((1_{\G},\ldots,1_{\G}), e)}).$$
Therefore, we obtain the required $\G\wr S_n$-graded $\G\wr S_n$-acted algebra map
$$\zeta_{\mathrm{wr}}:\C^{\otimes n}\to C_{A\wr S_n}((A\wr S_n)_{((1_{\G},\ldots,1_{\G}),e)}).$$
Indeed, given $((g_1,\ldots,g_n),\sigma)\in \G\wr S_n$ and $c_1\otimes\ldots\otimes c_n\in \C^{\otimes n}$ we have:
$$
\begin{array}{rcl}
\zeta_{\mathrm{wr}}(\tensor*[^{{((g_1,\ldots,g_n),\sigma)}}]{{(c_1\otimes\ldots\otimes c_n)}}{}) & = & \zeta_{\mathrm{wr}}(\tensor*[^{{g_{1}}}]{{c_{\sigma^{-1}(1)}}}{}\otimes\ldots\otimes \tensor*[^{{g_{n}}}]{{c_{\sigma^{-1}(n)}}}{})\\
 & = & \zeta(\tensor*[^{{g_{1}}}]{{c_{\sigma^{-1}(1)}}}{})\otimes\ldots\otimes \zeta(\tensor*[^{{g_{n}}}]{{c_{\sigma^{-1}(n)}}}{})\\
 & = & \tensor*[^{{g_{1}}}]{{\zeta(c_{\sigma^{-1}(1)})}}{}\otimes\ldots\otimes \tensor*[^{{g_{n}}}]{{\zeta(c_{\sigma^{-1}(n)})}}{}\\
 & = & \tensor*[^{{((g_1,\ldots,g_n),\sigma)}}]{{(\zeta(c_1)\otimes\ldots\otimes \zeta(c_n))}}{}\\
 & = & \tensor*[^{{((g_1,\ldots,g_n),\sigma)}}]{{\zeta_{\mathrm{wr}}(c_1\otimes\ldots\otimes c_n)}}{}.
\end{array}
$$
Henceforth, $A\wr S_n$ is a $\G\wr S_n$-graded algebra over $\C^{\otimes n}$.
\end{proof}
\medskip
\medskip
\medskip
\medskip
\section{Morita equivalences for wreath products}

Consider the notations from Section \ref{sec:Preliminaries} and Section \ref{sec:Wreath_products}. We recall the definition of a wreath product between a module and $S_n$.

\begin{definition}
Let $A$ and $A'$ be two algebras. Assume that $\tilde{M}$ is an $(A,A')$-bimodule. The wreath product $\tilde{M}\wr S_n$ is defined by
$$\tilde{M}\wr S_n:=\tilde{M}^{\otimes n}\otimes \O S_n$$
 as $\O$-modules, with  operations
$$
\begin{array}{l}
((a_1\otimes\ldots\otimes a_n)\otimes\sigma)((\tilde{m}_1\otimes\ldots\otimes \tilde{m}_n)\otimes\tau)\\
\qquad := ((a_1\otimes\ldots\otimes a_n)\cdot\tensor[^{\sigma}]{(\tilde{m}_1\otimes\ldots\otimes \tilde{m}_n)}{})\otimes(\sigma\tau),
\end{array}
$$
and
$$
\begin{array}{l}
((\tilde{m}_1\otimes\ldots\otimes \tilde{m}_n)\otimes\tau)((a'_1\otimes\ldots\otimes a'_n)\otimes\pi)\\
\qquad := ((\tilde{m}_1\otimes\ldots\otimes \tilde{m}_n)\cdot\tensor[^{\tau}]{(a'_1\otimes\ldots\otimes a'_n)}{})\otimes(\tau\pi),
\end{array}
$$
where
$$\tensor[^{\sigma}]{(\tilde{m}_1\otimes\ldots\otimes \tilde{m}_n)}{}:=\tilde{m}_{\sigma^{-1}(1)}\otimes\ldots\otimes \tilde{m}_{\sigma^{-1}(n)},$$
for all $(a_1\otimes\ldots\otimes a_n)\otimes\sigma\in A\wr S_n$, $(\tilde{m}_1\otimes\ldots\otimes \tilde{m}_n)\otimes\tau\in \tilde{M}\wr S_n$ and $(a'_1\otimes\ldots\otimes a'_n)\otimes\pi \in A'\wr S_n$.
\end{definition}

\begin{subsec}
Let $\C$ be a $\G$-graded $\G$-acted algebra and $A$ and $A'$ be two $\G$-graded crossed products over $\C$, with identity components $B$ and $B'$ respectively.

If $\tilde{M}$ is an $(A,A')$-bimodule which induces a Morita equivalence between $A$ and $A'$, by the results of   \cite[Section 5.1.C]{book:Marcus1999},  we already know that $\tilde{M}\wr S_n$ induces a Morita equivalence between $A\wr S_n$ and $A'\wr S_n$. The question that arises  is whether this result can be extended to give a graded Morita equivalence over a group graded group acted algebra.
\end{subsec}
\newpage
\begin{theorem} \label{th:main}
Let $\tilde{M}$ be a $\G$-graded $(A,A')$-bimodule over $\C$. Then, the following affirmations hold:
	\begin{enumerate}
		\item [(1)] $\tilde{M}\wr S_n$ is a $\G\wr S_n$-graded $(A\wr S_n,A'\wr S_n)$-bimodule over $\C^{\otimes n}$;
		\item [(2)] $(A\wr S_n)\otimes_{B^{\otimes n}}M^{\otimes n}\simeq M^{\otimes n}\otimes_{B'^{\otimes n}}(A'\wr S_n)\simeq \tilde{M}\wr S_n$ as $\G\wr S_n$-graded $(A\wr S_n,A'\wr S_n)$-bimodules over $\C^{\otimes n}$, where $M$ is the identity component of $\tilde{M}$;
		\item [(3)] If $\tilde{M}$ induces a $\G$-graded Morita equivalence over $\C$ between $A$ and $A'$, then $\tilde{M}\wr S_n$ induces a $\G\wr S_n$-graded Morita equivalence over $\C^{\otimes n}$ between $A\wr S_n$ and $A'\wr S_n$.
	\end{enumerate}
\end{theorem}

\begin{proof}
(1) By Lemma \ref{lemma:WR_algebra}, we know that $A\wr S_n$ and $A'\wr S_n$ are $\G\wr S_n$-graded crossed products over $\C^{\otimes n}$.

It is also known that $A\wr S_n$ and $A'\wr S_n$ are strongly $S_n$-graded algebras, and given this grading, if we denote
$$\Delta_{S_n}(A\wr S_n\otimes (A'\wr S_n)^{\mathrm{op}}):=(A\wr S_n\otimes (A'\wr S_n)^{\mathrm{op}})_{\delta(S_n)},$$
where $\delta(S_n):=\set{(\sigma,\sigma^{-1})\mid \sigma\in S_n}$, we have the following isomorphism of algebras:
$$\Delta_{S_n}(A\wr S_n\otimes (A'\wr S_n)^{\mathrm{op}})\simeq (A^{\otimes n}\otimes (A'^{\otimes n})^{\mathrm{op}})\otimes \O S_n.$$
Moreover,  \cite[Lemma 5.1.19]{book:Marcus1999} states that $\tilde{M}^{\otimes n}$ is a left $\O S_n$-module with the action given by permutations. Henceforth, it is easy to see that $\tilde{M}^{\otimes n}$ is a $(A^{\otimes n}\otimes (A'^{\otimes n})^{\mathrm{op}})\otimes \O S_n$-module, and thereby (the above isomorphism), $\tilde{M}^{\otimes n}$ extends to a $\Delta_{S_n}(A\wr S_n\otimes (A'\wr S_n)^{\mathrm{op}})$-module. Thus, $\tilde{M}\wr S_n:=\tilde{M}^{\otimes n}\otimes \O S_n$ becomes an $(A\wr S_n,A'\wr S_n)$-bimodule.

Now, we will prove that $\tilde{M}\wr S_n$ is a $\G\wr S_n$-graded $(A\wr S_n,A'\wr S_n)$-bimodule. Indeed, for all $((g_1,\ldots,g_n),\sigma)\in \G\wr S_n$, the $((g_1,\ldots,g_n),\sigma)$-component of $\tilde{M}\wr S_n$ is:
$$
(\tilde{M}\wr S_n)_{((g_1,\ldots,g_n),\sigma)}:=(\tilde{M}_{g_1}\otimes\ldots\otimes\tilde{M}_{g_n})\otimes\O\sigma.
$$
The verification for this definition is straightforward, as follows. For each $((g_1,\ldots,g_n),\sigma)$, $((x_1,\ldots,x_n),\pi)$ and $((h_1,\ldots,h_n),\tau)\in \G\wr S_n$ we have:\\
$
(A\wr S_n)_{((g_1,\ldots,g_n),\sigma)}(\tilde{M}\wr S_n)_{((x_1,\ldots,x_n),\pi)}(A'\wr S_n)_{((h_1,\ldots,h_n),\tau)}\\=
((A_{g_1}\otimes\ldots\otimes A_{g_n})\otimes\O\sigma)((\tilde{M}_{x_1}\otimes\ldots\otimes\tilde{M}_{x_n})\otimes\O\pi)((A'_{h_1}\otimes\ldots\otimes A'_{h_n})\otimes\O\tau)\\
=((A_{g_1}\tilde{M}_{x_{\sigma^{-1}(1)}}\otimes\ldots\otimes A_{g_n}\tilde{M}_{x_{\sigma^{-1}(n)}})\otimes\O(\sigma\pi))((A'_{h_1}\otimes\ldots\otimes A'_{h_n})\otimes\O\tau)\\
=(A_{g_1}\tilde{M}_{x_{\sigma^{-1}(1)}}A'_{h_{(\sigma\pi)^{-1}(1)}}\otimes\ldots\otimes A_{g_n}\tilde{M}_{x_{\sigma^{-1}(n)}}A'_{h_{(\sigma\pi)^{-1}(n)}})\otimes\O(\sigma\pi\tau)\\
\subseteq(\tilde{M}_{g_1 x_{\sigma^{-1}(1)} h_{(\sigma\pi)^{-1}(1)}}\otimes\ldots\otimes \tilde{M}_{g_n x_{\sigma^{-1}(n)} h_{(\sigma\pi)^{-1}(n)}})\otimes\O(\sigma\pi\tau)\\
=(\tilde{M}\wr S_n)_{((g_1 x_{\sigma^{-1}(1)} h_{(\sigma\pi)^{-1}(1)},\ldots, g_n x_{\sigma^{-1}(n)} h_{(\sigma\pi)^{-1}(n)}),\sigma\pi\tau)}\\
=(\tilde{M}\wr S_n)_{((g_1 x_{\sigma^{-1}(1)},\ldots, g_n x_{\sigma^{-1}(n)}),\sigma\pi)((h_{1},\ldots, h_{n}),\tau)}\\
=(\tilde{M}\wr S_n)_{((g_1,\ldots, g_n),\sigma)((x_{1},\ldots, x_{n}),\pi)((h_{1},\ldots, h_{n}),\tau)}.
$\\
Therefore, $\tilde{M}\wr S_n$ is a $\G\wr S_n$-graded $(A\wr S_n,A'\wr S_n)$-bimodule. Note that the identity component of $\tilde{M}\wr S_n$ (with respect to the $\G\wr S_n$-grading) is $$(\tilde{M}\wr S_n)_1=M^{\otimes n},$$
where $M$ is the identity component of $\tilde{M}$.

Finally, we will prove that
$$
\begin{array}{l}
((\tilde{m}_{g_1}\otimes\ldots\otimes\tilde{m}_{g_n})\otimes \sigma)(c_1\otimes\ldots\otimes c_n)\\
\qquad\qquad = \tensor*[^{((g_1,\ldots,g_n),\sigma)}]{{(c_1\otimes\ldots\otimes c_n)}}{}((\tilde{m}_{g_1}\otimes\ldots\otimes\tilde{m}_{g_n})\otimes \sigma),
\end{array}
$$
for all $(\tilde{m}_{g_1}\otimes\ldots\otimes\tilde{m}_{g_n})\otimes \sigma\in(\tilde{M}\wr S_n)_{((g_1,\ldots,g_n),\sigma)}$ and $c_1\otimes\ldots\otimes c_n\in\C^{\otimes n}$ and for all $((g_1,\ldots,g_n),\sigma)\in\G\wr S_n$. Indeed,
$$
\begin{array}{l}
((\tilde{m}_{g_1}\otimes\ldots\otimes\tilde{m}_{g_n})\otimes \sigma)(c_1\otimes\ldots\otimes c_n)\\
\qquad =  (\tilde{m}_{g_1}c_{\sigma^{-1}(1)}\otimes\ldots\otimes\tilde{m}_{g_n}c_{\sigma^{-1}(n)})\otimes \sigma\\
\qquad =  (\tensor*[^{{g_1}}]{{c_{\sigma^{-1}(1)}}}{}\tilde{m}_{g_1}\otimes\ldots\otimes \tensor*[^{{g_n}}]{{c_{\sigma^{-1}(n)}}}{}\tilde{m}_{g_n})\otimes \sigma\\
\qquad =  (\tensor*[^{{g_1}}]{{c_{\sigma^{-1}(1)}}}{}\otimes\ldots\otimes\tensor*[^{{g_n}}]{{c_{\sigma^{-1}(n)}}}{})((\tilde{m}_{g_1}\otimes\ldots\otimes \tilde{m}_{g_n})\otimes \sigma)\\
\qquad = \tensor*[^{((g_1,\ldots,g_n),\sigma)}]{{(c_1\otimes\ldots\otimes c_n)}}{}((\tilde{m}_{g_1}\otimes\ldots\otimes\tilde{m}_{g_n})\otimes \sigma).
\end{array}
$$
Henceforth, $\tilde{M}\wr S_n$ is a $\G\wr S_n$-graded $(A\wr S_n,A'\wr S_n)$-bimodule over $\C^{\otimes n}$.

(2) In this part, in order to prove our claim, we want to use a similar technique as in part (1), but with regard to the grading given by $\G\wr S_n$.

Henceforth, we regard $(A'\wr S_n)^{\mathrm{op}}$ as a $\G\wr S_n$-graded crossed product over $\C^{\otimes n}$, where the $((g_1,\ldots,g_n),\sigma)$-component of $(A'\wr S_n)^{\mathrm{op}}$ is:
$$(A'\wr S_n)^{\mathrm{op}}_{((g_1,\ldots,g_n),\sigma)}:=((A'\wr S_n)_{((g_1,\ldots,g_n),\sigma)^{-1}})^{\mathrm{op}},$$
for all $((g_1,\ldots,g_n),\sigma)\in\G\wr S_n$, and we consider:
$$
\begin{array}{l}
\Delta^{\C^{\otimes n}}_{\G\wr S_n}(A\wr S_n\otimes_{\C^{\otimes n}} (A'\wr S_n)^{\mathrm{op}}):=\\
\qquad\bigoplus_{((g_1,\ldots,g_n),\sigma)\in \G\wr S_n}((A\wr S_n)_{((g_1,\ldots,g_n),\sigma)}\otimes_{\C^{\otimes n}} (A'\wr S_n)^{\mathrm{op}}_{((g_1,\ldots,g_n),\sigma)}),
\end{array}
$$
which, by \cite[Lemma 2.8]{art:MM2}, is an $\O$-algebra.

Now, given the fact from part (1) of this theorem, that $\tilde{M}\wr S_n$ is a $\G\wr S_n$-graded $(A\wr S_n,A'\wr S_n)$-bimodule over $\C^{\otimes n}$, we obtain, by  \cite[Proposition 2.11]{art:MM2}, that $(\tilde{M}\wr S_n)_1=M^{\otimes n}$ extends to a $\Delta^{\C^{\otimes n}}_{\G\wr S_n}(A\wr S_n\otimes_{\C^{\otimes n}} (A'\wr S_n)^{\mathrm{op}})$-module and that we have the following isomorphisms of $\G\wr S_n$-graded $(A\wr S_n,A'\wr S_n)$-bimodules over $\C^{\otimes n}$:
$$(A\wr S_n)\otimes_{B^{\otimes n}}M^{\otimes n}\simeq M^{\otimes n}\otimes_{B'^{\otimes n}}(A'\wr S_n)\simeq \tilde{M}\wr S_n.$$
More exactly, these isomorphisms are:
$$f:(A\wr S_n)\otimes_{B^{\otimes n}}M^{\otimes n}\to \tilde{M}\wr S_n, \quad (a\otimes\sigma)\otimes m\mapsto (a\cdot \tensor*[^{\sigma}]{m}{})\otimes\sigma,$$
and
$$g:M^{\otimes n}\otimes_{B'^{\otimes n}}(A'\wr S_n)\to \tilde{M}\wr S_n,  \quad m\otimes(a'\otimes\sigma) \mapsto (m\cdot a')\otimes\sigma$$
for all $a\in A^{\otimes n}$, $a'\in A'^{\otimes n}$, $m\in M^{\otimes n}$ and $\sigma \in S_n$.

We prove that $f$ is an isomorphism of $\G\wr S_n$-graded $(A\wr S_n,A'\wr S_n)$-bimodules over $\C^{\otimes n}$. The verification for $g$ is similar. 
The left $A\wr S_n$-module structure of $(A\wr S_n)\otimes_{B^{\otimes n}}M^{\otimes n}$ is clear. We recall and particularize from \cite[Proposition 2.11]{art:MM2} the right $A'\wr S_n$-module structure of $(A\wr S_n)\otimes_{B^{\otimes n}}M^{\otimes n}$:
\begin{align*}
&((a\otimes \sigma)\otimes_{B^{\otimes n}} m)\cdot (a'_g\otimes \tau)\\
&\quad = ((a\otimes \sigma)(u_g\otimes \tau)) \otimes_{B^{\otimes n}} (((\tensor*[^{\tau^{-1}}]{u}{^{-1}_{g}}\otimes \tau^{-1})\otimes_{\C^{\otimes n}}(a'_g\otimes \tau)^{\op})m)\\
&\quad = (a \cdot\tensor*[^{\sigma}]{u}{_g}\otimes \sigma\tau) \otimes_{B^{\otimes n}} ((\tensor*[^{\tau^{-1}}]{u}{^{-1}_{g}}\otimes \tau^{-1})(m\otimes e)(a'_g\otimes \tau))\\
&\quad = (a \cdot\tensor*[^{\sigma}]{u}{_g}\otimes \sigma\tau) \otimes_{B^{\otimes n}} ((\tensor*[^{\tau^{-1}}]{u}{^{-1}_{g}}\cdot\tensor*[^{\tau^{-1}}]{m}{}\otimes \tau^{-1})(a'_g\otimes \tau))\\
&\quad = (a \cdot\tensor*[^{\sigma}]{u}{_g}\otimes \sigma\tau) \otimes_{B^{\otimes n}} (\tensor*[^{\tau^{-1}}]{u}{^{-1}_{g}}\cdot\tensor*[^{\tau^{-1}}]{m}{}\cdot\tensor*[^{\tau^{-1}}]{{a'_g}}{}),
\end{align*}
for all $a\in A^{\otimes n}$, $\sigma,\tau\in S_n$, $m\in M^{\otimes n}$, $a'_g\in A'^{\otimes n}_g$, $g\in \G^n$ and $u_g$ is an invertible homogeneous element of $A^{\otimes n}_g$.
	We start by proving that $f$ is a morphism of $(A\wr S_n,A'\wr S_n)$-bimodules:
\begin{align*}
f((b\otimes \tau)\cdot((a\otimes \sigma)\otimes_{B^{\otimes n}} m)) & =  f(((b\otimes \tau)\cdot(a\otimes \sigma))\otimes_{B^{\otimes n}} m)\\
& =  f((b\cdot\tensor*[^{\tau}]{a}{}\otimes \tau\sigma)\otimes_{B^{\otimes n}} m)\\
& =   b\cdot\tensor*[^{\tau}]{a}{}\cdot \tensor*[^{\tau\sigma}]{m}{}\otimes \tau\sigma\\
& =   b\cdot\tensor*[^{\tau}]{{(a\cdot \tensor*[^{\sigma}]{m}{})}}{}\otimes \tau\sigma\\
& =  (b\otimes\tau)(a\cdot \tensor*[^{\sigma}]{m}{}\otimes \sigma)\\
& =  (b\otimes\tau)f((a\otimes \sigma)\otimes_{B^{\otimes n}} m),
\end{align*}
\begin{align*}
&f(((a\otimes \sigma)\otimes_{B^{\otimes n}} m)\cdot (a'_g\otimes \tau))\\
&\quad = f((a \cdot\tensor*[^{\sigma}]{u}{_g}\otimes \sigma\tau) \otimes_{B^{\otimes n}} (\tensor*[^{\tau^{-1}}]{u}{^{-1}_{g}}\cdot\tensor*[^{\tau^{-1}}]{m}{}\cdot\tensor*[^{\tau^{-1}}]{{a'_g}}{}))\\
&\quad = (a \cdot\tensor*[^{\sigma}]{u}{_g} \cdot \tensor*[^{\sigma\tau}]{{(\tensor*[^{\tau^{-1}}]{u}{^{-1}_{g}}\cdot\tensor*[^{\tau^{-1}}]{m}{}\cdot\tensor*[^{\tau^{-1}}]{{a'_g}}{})}}{})\otimes \sigma\tau\\
&\quad = (a \cdot\tensor*[^{\sigma}]{u}{_g} \cdot \tensor*[^{\sigma\tau\tau^{-1}}]{u}{^{-1}_{g}}\cdot\tensor*[^{\sigma\tau\tau^{-1}}]{m}{}\cdot\tensor*[^{\sigma\tau\tau^{-1}}]{{a'_g}}{})\otimes \sigma\tau\\
&\quad = (a \cdot\tensor*[^{\sigma}]{u}{_g} \cdot \tensor*[^{\sigma}]{u}{^{-1}_{g}}\cdot\tensor*[^{\sigma}]{m}{}\cdot\tensor*[^{\sigma}]{{a'_g}}{})\otimes \sigma\tau\\
&\quad = (a \cdot\tensor*[^{\sigma}]{m}{}\cdot\tensor*[^{\sigma}]{{a'_g}}{})\otimes \sigma\tau\\
&\quad = ((a \cdot\tensor*[^{\sigma}]{m}{})\otimes \sigma)(a'_g\otimes \tau)\\
&\quad = f((a\otimes \sigma)\otimes_{B^{\otimes n}} m)(a'_g\otimes \tau),
\end{align*}
for all $a,b\in A^{\otimes n}$, $\sigma,\tau\in S_n$, $m\in M^{\otimes n}$, $a'_g\in A'^{\otimes n}_g$, $g\in \G^n$ and $u_g$ is an invertible homogeneous element of $A^{\otimes n}_g$.

Next, we will prove that $f$ is $\G\wr S_n$-grade preserving. We recall from \cite[Proposition 2.11]{art:MM2} that the $\G\wr S_n$-grading of $(A\wr S_n)\otimes_{B^{\otimes n}}M^{\otimes n}$ is given by $A\wr S_n$. We have:
$$
\begin{array}{l}
f(((a_{g_1}\otimes\ldots \otimes a_{g_n})\otimes\sigma)\otimes_{B^{\otimes n}}(m_1\otimes \ldots\otimes m_n))\\ \quad = ((a_{g_1}m_{\sigma^{-1}(1)}\otimes\ldots\otimes a_{g_n}m_{\sigma^{-1}(n)})\otimes\sigma\\
\quad \in ((A_{g_1}\tilde{M}_1\otimes\ldots\otimes A_{g_n}\tilde{M}_1)\otimes\O\sigma\\
\quad \subseteq ((\tilde{M}_{g_1}\otimes\ldots\otimes\tilde{M}_{g_n})\otimes\O\sigma\\
\quad = (\tilde{M}\wr S_n)_{((g_1,\ldots,g_n),\sigma)},
\end{array}
$$
for all $(a_{g_1}\otimes\ldots \otimes a_{g_n})\otimes \sigma\in (A\wr S_n)_{((g_1,\ldots,g_n),\sigma)}$, $m_1\otimes \ldots\otimes m_n\in M^{\otimes n}$ and for all $((g_1,\ldots,g_n),\sigma)\in \G\wr S_n$.

Finally, we will prove that $f$ is bijective. Because both modules have the same $\O$-rank, it is enough to prove that $f$ is surjective. If $(\tilde{m}_{g_1}\otimes\ldots\otimes\tilde{m}_{g_n})\otimes \sigma$ is an arbitrary element of $(\tilde{M}\wr S_n)_{((g_1,\ldots,g_n),\sigma)}$, then it is clear that 
$$((u^{-1}_{g^{-1}_1}\otimes\ldots\otimes u^{-1}_{g^{-1}_n})\otimes \sigma)\otimes (u_{g^{-1}_{\sigma(1)}}\tilde{m}_{g_{\sigma(1)}}\otimes\ldots\otimes u_{g^{-1}_{\sigma(n)}}\tilde{m}_{g_{\sigma(n)}})\in (A\wr S_n)\otimes_{B^{\otimes n}}M^{\otimes n}$$ 
and that
$$
\begin{array}{l}
f(((u^{-1}_{g^{-1}_1}\otimes\ldots\otimes u^{-1}_{g^{-1}_n})\otimes \sigma)\otimes (u_{g^{-1}_{\sigma(1)}}\tilde{m}_{g_{\sigma(1)}}\otimes\ldots\otimes u_{g^{-1}_{\sigma(n)}}\tilde{m}_{g_{\sigma(n)}}))\\
\quad = ((u^{-1}_{g^{-1}_1}\otimes\ldots\otimes u^{-1}_{g^{-1}_n})\cdot \tensor*[^{\sigma}]{{(u_{g^{-1}_{\sigma(1)}}\tilde{m}_{g_{\sigma(1)}}\otimes\ldots\otimes u_{g^{-1}_{\sigma(n)}}\tilde{m}_{g_{\sigma(n)}})}}{})\otimes \sigma \\
\quad = ((u^{-1}_{g^{-1}_1}\otimes\ldots\otimes u^{-1}_{g^{-1}_n})\cdot(u_{g^{-1}_{1}}\tilde{m}_{g_{1}}\otimes\ldots\otimes u_{g^{-1}_{n}}\tilde{m}_{g_{n}}))\otimes \sigma \\
\quad = (u^{-1}_{g^{-1}_1}u_{g^{-1}_{1}}\tilde{m}_{g_{1}}\otimes\ldots\otimes u^{-1}_{g^{-1}_n}u_{g^{-1}_{n}}\tilde{m}_{g_{n}})\otimes \sigma \\
\quad = (\tilde{m}_{g_{1}}\otimes\ldots\otimes \tilde{m}_{g_{n}})\otimes \sigma,
\end{array}
$$
for all $((g_1,\ldots,g_n),\sigma)\in \G\wr S_n$, where $u_{g}$ represents an invertible homogeneous element of $A_{g}$, for all $g\in\G$.

(3) Furthermore, by Proposition \ref{prop:tensorProduct_Morita}, $\tilde{M}^{\otimes n}$ is a $\G^n$-graded $(A^{\otimes n},A'^{\otimes n})$-bimodule over $\C^{\otimes n}$, which induces a $\G^n$-graded Morita equivalence over $\C^{\otimes n}$ between $A^{\otimes n}$ and $A^{\otimes n}$, thus by  \cite[Theorem 5.1.2]{book:Marcus1999} with respect to the $\G^n$-grading, we have that $(\tilde{M}^{\otimes n})_1={M}^{\otimes n}$ is a $(B^{\otimes n},B'^{\otimes n})$-bimodule, which induces a Morita equivalence between $B^{\otimes n}$ and $B'^{\otimes n}$.

Now, by the previous statements, and by using \cite[Theorem 3.3]{art:MM2} with respect to the $\G\wr S_n$-grading, we obtain that $\tilde{M}\wr S_n$ induces a $\G\wr S_n$-graded Morita equivalence over $\C^{\otimes n}$  between $A\wr S_n$ and $A'\wr S_n$. 
\end{proof}
\medskip
\medskip
\medskip
\medskip
\medskip
\medskip
\medskip
\medskip
\medskip
\medskip

\phantomsection

\end{document}